\documentclass{article}
\usepackage{amssymb}
\usepackage{amsmath}
\usepackage{geometry}
\usepackage{lineno}

\newtheorem{theorem}{Theorem}[section]

\newtheorem{corollary}[theorem]{Corollary}

\newtheorem{example}[theorem]{Example}

\newtheorem{lemma}[theorem]{Lemma}

\newtheorem{remark}[theorem]{Remark}

\newenvironment{proof}[1][Proof]{\noindent\textbf{#1.} }{\ \rule{0.5em}{0.5em}}
\input{tcilatex}
\begin{document}

\begin{center}
\textbf{Boundedness\ of\ Fractional\ Integral\ operators\ and their\
commutators in vanishing\ generalized weighted\ Morrey\ spaces}

\bigskip

Bilal \c{C}eki\c{c}$^{a,}\footnote{%
Corresponding author. E-mail address: bilalc@dicle.edu.tr}$ and Ay\c{s}eg%
\"{u}l \c{C}elik Alabal\i k$^{b}$

\bigskip

\textbf{Abstract}
\end{center}

\begin{quote}
In this article, we establish some conditions for the boundedness of
fractional integral operators on the vanishing generalized weighted Morrey
spaces. We also investigate corresponding commutators generated by BMO
functions.

\textbf{Keywords :} Vanishing generalized weighted Morrey spaces, fractional
integral operator, Commutator

\textbf{MSC \ \ \ \ \ \ \ :} 42B20; 42B35; 46E30
\end{quote}

\section{\textbf{Introduction}}

Morrey spaces $L^{p,\lambda }$ that play important role in the theory of
partial differential equations an in harmonic analysis were introduced by C.
Morrey \cite{Morrey} in 1938. Since then these spaces and various
generalizations of Morrey spaces have been extensively studied by many
authors. Mizuhara \cite{Mizuhara} introduced the generalized Morrey space $%
L^{p,\varphi }$ and Komori and Shirai \cite{Komori} defined the weighted
Morrey spaces $L^{p,\kappa }(w)$. Guliyev \cite{guli2012} have given the
notion of generalized weighted Morrey space $L^{p,\varphi }\left( w\right) $
which can be accepted as an extension of $L^{p,\varphi }$ and $L^{p,\kappa
}(w)$. We refer readers to the survey \cite{Rafe} and to the elegant book 
\cite{Adams2} for further generalizations about these spaces and references
on recent developments in this field.

Fractional maximal operator $M_{\alpha }$ and fractional integral operator $%
I_{\alpha }$\ play an important role in harmonic analysis. In recent years
many authors have studied the boundedness of these operators on Morrey type
spaces. The boundedness of fractional integral operator $I_{\alpha }$ was
proved by Adams \cite{Adams} on classical Morrey spaces. In \cite{Guliyev4}
authours find conditions on the pair $\left( \varphi _{1},\varphi
_{2}\right) $ which ensure Spanne type boundedness of fractional maximal $%
M_{\alpha }$ and fractional integral operator $I_{\alpha }$ from one
generalized Morrey spaces to another generalized Morrey space. The
boundedness of fractional integral operator $I_{\alpha }$ was proved by
Komori and Shirai \cite{Komori}\ on weighted Morrey spaces. The boundedness
of $I_{\alpha ,b}$ was proved by Di Fazio and Ragusa \cite{Ragusa 2}\ on
classical Morrey spaces, by Guliyev and Shukurov \cite{Guliyev4} on
generalized Morrey space and by Komori and Shirai \cite{Komori}\ on weighted
Morrey spaces. Recently, Guliyev proved the boundedness of fractional
integral operator $I_{\alpha }$ and  fractional maximal operator $M_{\alpha }
$ and their commutators on generalized weighted Morrey space $L^{p,\varphi
}\left( w\right) $ in the elegant paper \cite{guli2012}.

The vanishing Morrey space $VL^{p,\lambda }$ which has been first introduced
by Vitanza in \cite{Vitanza} is a subspace of functions in $L^{p,\lambda }$
satisfying the following condition

\begin{equation*}
\underset{r\rightarrow 0}{\lim }\sup\limits_{\substack{ x\in 
%TCIMACRO{\U{211d} }%
%BeginExpansion
\mathbb{R}
%EndExpansion
^{n}  \\ 0<t<r}}t^{-\frac{\lambda }{p}}\left\Vert f\right\Vert _{L^{p}\left(
B\left( x,t\right) \right) }=0.
\end{equation*}

Ragusa \cite{Ragusa} proved the boundedness of $I_{\alpha ,b}$ on vanishing
Morrey spaces. Persson et al. \cite{Persson} proved the boundedness of
commutators of Hardy operators in vanishing Morrey spaces. Later N. Samko 
\cite{Samko} introduced the vanishing generalized Morrey spaces $VL_{\Pi
}^{p,\varphi }\left( \Omega \right) $\textrm{\ }and proved the boundedness
of $I_{\alpha }$ and $M_{\alpha \text{ }}$ on these spaces.

Inspired by the above papers, we introduce the vanishing generalized
weighted Morrey spaces which are suitable subspace of functions in $%
L^{p,\varphi }\left( w\right) $ and prove the boundedness of $M_{\alpha 
\text{ }}$and $I_{\alpha }$ from the vanishing generalized weighted Morrey
space $VL_{\Pi }^{p,\varphi }\left( 
%TCIMACRO{\U{211d} }%
%BeginExpansion
\mathbb{R}
%EndExpansion
^{n};w^{p}\right) $ to another vanishing generalized weighted Morrey space $%
VL_{\Pi }^{q,\psi }\left( 
%TCIMACRO{\U{211d} }%
%BeginExpansion
\mathbb{R}
%EndExpansion
^{n};w^{q}\right) .$ Furthermore we show that the commutators of both $%
M_{\alpha \text{ }}$and $I_{\alpha }$ are bounded in vanishing generalized
weighted Morrey spaces. In all the cases the conditions for the boundedness
are given in terms of Zygmund-type integral inequalities on $\left( \varphi
,\psi \right) $, where there is no any assumption on monotonicity of $%
\varphi $, $\psi $ in $r$.

Throughout this paper, $C,c,c_{i}$ etc. are used as positive constant that
can change from one line to another. $A\lesssim B$ means that $A\leq cB$
with some positive constant $c$. If $A\lesssim B$ and $B\lesssim A$, then we
say $A\approx B$ which means $A\ $and $B$ are equivalent.

\section{Preliminaries}

Let $w$ be a weight function on $%
%TCIMACRO{\U{211d} }%
%BeginExpansion
\mathbb{R}
%EndExpansion
^{n},$ such that $w(x)>0$ for almost every $x\in 
%TCIMACRO{\U{211d} }%
%BeginExpansion
\mathbb{R}
%EndExpansion
^{n}.$ For $1\leq p<\infty $ and $\Omega \subset 
%TCIMACRO{\U{211d} }%
%BeginExpansion
\mathbb{R}
%EndExpansion
^{n}$\ we denote the weighted Lebesgue space by $L^{p,w}\left( \Omega
\right) $ with the norm

\begin{equation*}
\left\Vert f\right\Vert _{L^{p,w}\left( \Omega \right) }=\left(
\int\limits_{\Omega }\left\vert f(x)\right\vert ^{p}w(x)dx\right)
^{1/p}<\infty .
\end{equation*}

Let $1\leq p<\infty $, $\varphi $ be a positive measurable function on $%
%TCIMACRO{\U{211d} }%
%BeginExpansion
\mathbb{R}
%EndExpansion
^{n}\times \left( 0,\infty \right) $ and $w$ be a weight function on $%
%TCIMACRO{\U{211d} }%
%BeginExpansion
\mathbb{R}
%EndExpansion
^{n}$. We denote by $L^{p,\varphi }\left( 
%TCIMACRO{\U{211d} }%
%BeginExpansion
\mathbb{R}
%EndExpansion
^{n};w\right) =$ $L^{p,\varphi }\left( w\right) $ the generalized weighted
Morrey space, the space of all functions $f\in L_{loc}^{p,w}\left( 
%TCIMACRO{\U{211d} }%
%BeginExpansion
\mathbb{R}
%EndExpansion
^{n}\right) $ with finite quasinorm

\begin{equation*}
\left\Vert f\right\Vert _{L^{p,\varphi }\left( w\right) }=\underset{x\in 
%TCIMACRO{\U{211d} }%
%BeginExpansion
\mathbb{R}
%EndExpansion
^{n},r>0}{\sup }\frac{1}{\varphi ^{\frac{1}{p}}\left( x,r\right) }\left\Vert
f\right\Vert _{L^{p,w}\left( B\left( x,r\right) \right) }
\end{equation*}%
where $B\left( x,r\right) $ denote the open ball centered at $x$ of radius $%
r.$

The fractional integral operator (Riesz potential) $I_{\alpha }$ and
fractional Maximal operator $M_{\alpha },$ which play important roles in
real and harmonic analysis, are defined by%
\begin{equation*}
I_{\alpha }f(x)=\dint\limits_{%
%TCIMACRO{\U{211d} }%
%BeginExpansion
\mathbb{R}
%EndExpansion
^{n}}\frac{f(y)}{\left\vert x-y\right\vert ^{n-\alpha }}dy,\qquad 0<\alpha <n
\end{equation*}%
and

\begin{equation*}
M_{\alpha }f(x)=\underset{r>0}{\sup }\frac{1}{r^{n-\alpha }}%
\dint\limits_{B(x,r)}\left\vert f(y)\right\vert dy,\qquad 0\leq \alpha <n
\end{equation*}%
where $f\in L_{loc}^{1}\left( 
%TCIMACRO{\U{211d} }%
%BeginExpansion
\mathbb{R}
%EndExpansion
^{n}\right) $. If $\alpha =0,$ then $M\equiv M_{0}$ is the Hardy-Littlewood
maximal operator.

The class of $A_{p}$ weights was introduced by B. Muckenhoupt in \cite{Muc1}%
. The Muckenhoupt classes characterize the boundedness of the
Hardy-Littlewood maximal function $M$ on weighted Lebesgue spaces. It is
known that, $M$ is bounded on $L^{p,w}(%
%TCIMACRO{\U{211d} }%
%BeginExpansion
\mathbb{R}
%EndExpansion
^{n})$ if and only if $w\in A_{p}$, $1<p<\infty .$

Now we define Muckenhoupt class $A_{p}.$ Let $1<p<\infty $ and $\left\vert
B\right\vert $ be the Lebesgue measure of the ball $B.$ A weight $w$ is said
to be an $A_{p}$ weight, if there exists a positive constant $c_{p}$ such
that, for every ball $B\subset 
%TCIMACRO{\U{211d} }%
%BeginExpansion
\mathbb{R}
%EndExpansion
^{n},$

\begin{equation*}
\left( \dint\limits_{B}w\left( x\right) dx\right) \left(
\dint\limits_{B}w\left( x\right) ^{1-p^{\prime }}dx\right) ^{p-1}\leq
c_{p}\left\vert B\right\vert ^{p},
\end{equation*}%
when $1<p<\infty $, and for $p=1$

\begin{equation*}
\left( \dint\limits_{B}w\left( y\right) dy\right) \leq c_{1}\left\vert
B\right\vert w\left( x\right)
\end{equation*}%
for a.e. $x\in B.$

A weight $w$ belongs to $A_{p,q}$ for $1<p<q<\infty $ if there exists $%
c_{p,q}>1$ such that

\begin{equation*}
\left\Vert w\right\Vert _{L^{q}(B)}\left\Vert w^{-1}\right\Vert
_{L^{p^{\prime }}(B)}\leq c_{p,q}\left\vert B\right\vert ^{1+\frac{1}{q}-%
\frac{1}{p}},
\end{equation*}%
where $\frac{1}{p}+\frac{1}{p^{\prime }}=1$. $A_{p,q}$ class was introduce
by B. Muckenhoupt and R. Wheeden \cite{Muc2} to study weighted norm
inequalities for fractional integral operators.

Following lemma gives the relation between $A_{p,q}$ and $A_{p}$ class.

\begin{lemma}
\cite{Lu} If $w^{q}\in A_{1+\frac{q}{p^{\prime }}}$ with $1<p<q,$ then $w\in
A_{p,q}.$
\end{lemma}

\begin{lemma}
\cite{guli2012} Let $0<\alpha <n,$ $1<p<\frac{n}{\alpha },$ $\frac{1}{q}=%
\frac{1}{p}-\frac{\alpha }{n}$ and $w\in A_{p,q}$. Then the inequality%
\begin{equation}
\left\Vert I_{\alpha }f\right\Vert _{L^{q,w^{q}}\left( B\left( x,r\right)
\right) }\lesssim \left\Vert w\right\Vert _{L^{q}\left( B(x,r)\right)
}\dint\limits_{2r}^{\infty }\left\Vert f\right\Vert _{L^{p,w^{p}}\left(
B\left( x,t\right) \right) }\left\Vert w\right\Vert _{L^{q}\left( B\left(
x_{,}t\right) \right) }^{-1}\frac{dt}{t}  \tag{$2.1$}
\end{equation}%
holds for any ball $B(x,r)$ and for all $f\in L_{loc}^{p,w}\left( 
%TCIMACRO{\U{211d} }%
%BeginExpansion
\mathbb{R}
%EndExpansion
^{n}\right) .$
\end{lemma}

\subsection{\textbf{Vanishing generalized weighted Morrey spaces}}

Let $\Omega $ be an open set in $%
%TCIMACRO{\U{211d} }%
%BeginExpansion
\mathbb{R}
%EndExpansion
^{n}$ and $\Pi $ be an arbitrary subset of $\Omega $. Let also $w$ be a
weight function on $\Omega $, $\varphi \left( x,r\right) $ be a measurable
non negative function on $\Pi \times \left[ 0,l\right) $ $\left(
l=diam~\Omega \right) $ and positive for all $(x,t)\in \Pi \times (0,l)$.
The vanishing weighted Morrey space $VL_{\Pi }^{p,\varphi }\left( \Omega
;w\right) =VL_{\Pi }^{p,\varphi }\left( w\right) $ is defined as the space
of functions $f\in L_{loc}^{p,w}\left( \Omega \right) $ with finite quasi
norm

\begin{equation*}
\left\Vert f\right\Vert _{VL_{\Pi }^{p,\varphi }\left( w\right) }=\underset{%
x\in \Pi ,0<r<l}{\sup }\frac{1}{\varphi ^{\frac{1}{p}}\left( x,r\right) }%
\left\Vert f\right\Vert _{L^{p,w}\left( \widetilde{B}\left( x,r\right)
\right) }
\end{equation*}%
such that

\begin{equation}
\underset{r\rightarrow 0}{\lim }\text{ }\underset{x\in \Pi }{\sup }\frac{1}{%
\varphi ^{\frac{1}{p}}\left( x,r\right) }\left\Vert f\right\Vert
_{L^{p,w}\left( \widetilde{B}\left( x,r\right) \right) }=0,  \tag{$2.2$}
\label{21}
\end{equation}%
where $\widetilde{B}\left( x,r\right) =B\left( x,r\right) \cap \Omega $ and $%
1\leq p<\infty .$

Naturally, it is suitable to impose the function $\varphi \left( x,r\right) $
on the following conditions

\begin{equation}
\underset{r\rightarrow 0}{\lim }\text{ }\underset{x\in \Pi }{\sup }\frac{%
\left\Vert w\right\Vert _{L^{1}\left( \widetilde{B}(x,r)\right) }}{\varphi
\left( x,r\right) }=0  \tag{$2.3$}  \label{22}
\end{equation}%
and

\begin{equation}
\underset{r>1}{\inf }\text{ }\underset{x\in \Pi }{\sup }\text{ }\varphi
\left( x,r\right) >0  \tag{$2.4$}  \label{23}
\end{equation}%
where the last condition must be imposed when $\Omega $ is unbounded. These
conditions makes $VM_{\Pi }^{p,\varphi }\left( w\right) $\ non trivial,
since bounded functions which have compact support belong to these spaces.

Henceforth we denote by $\varphi \in \mathfrak{B}\left( w\right) $ if $%
\varphi \left( x,r\right) $ is a measurable non negative function on $\Pi
\times \left[ 0,l\right) $ and positive for all $(x,t)\in \Pi \times (0,l)$
and satisfies the condition (2.3) and (2.4).

\subsection{\textbf{Commutator of Fractional Integral Operator }}

Let $f\in L_{loc}^{1}\left( 
%TCIMACRO{\U{211d} }%
%BeginExpansion
\mathbb{R}
%EndExpansion
^{n}\right) $ we define%
\begin{equation*}
\left\Vert f\right\Vert _{\ast }=\underset{x\in 
%TCIMACRO{\U{211d} }%
%BeginExpansion
\mathbb{R}
%EndExpansion
^{n},r>0}{\sup }\frac{1}{\left\vert B\left( x,r\right) \right\vert }%
\dint\limits_{B\left( x,r\right) }\left\vert f\left( y\right) -f_{B\left(
x,r\right) }\right\vert dy
\end{equation*}%
where%
\begin{equation*}
f_{B\left( x,r\right) }=\frac{1}{\left\vert B\left( x,r\right) \right\vert }%
\dint\limits_{B\left( x,r\right) }f\left( y\right) dy.
\end{equation*}

The space of functions of bounded mean oscillation (BMO) are the class of
functions whose deviation from their means over cubes is bounded. The set%
\begin{equation*}
BMO\left( 
%TCIMACRO{\U{211d} }%
%BeginExpansion
\mathbb{R}
%EndExpansion
^{n}\right) =\left\{ f\in L_{loc}^{1}\left( 
%TCIMACRO{\U{211d} }%
%BeginExpansion
\mathbb{R}
%EndExpansion
^{n}\right) :\left\Vert f\right\Vert _{\ast }<\infty \right\} 
\end{equation*}%
is called the function space of BMO space.

Next we shall introduce the commutators of maximal operator $M_{\alpha ,b}$
and the commutator of fractional integral operator $I_{\alpha ,b}$.

Let $b$ be a locally integrable function, $M_{\alpha ,b}$ and $I_{\alpha ,b}$
which are formed by $b$ are defined by%
\begin{equation*}
M_{\alpha ,b}f\left( x\right) =b\left( x\right) M_{\alpha }f\left( x\right)
-M_{\alpha }\left( bf\right) \left( x\right)
\end{equation*}%
and%
\begin{equation*}
I_{\alpha ,b}f\left( x\right) =b\left( x\right) I_{\alpha }f\left( x\right)
-I_{\alpha }\left( bf\right) \left( x\right) .
\end{equation*}

\begin{lemma}
\cite{guli2012}

Let $0<\alpha <n,$ $1<p<\frac{n}{\alpha },$ $\frac{1}{q}=\frac{1}{p}-\frac{%
\alpha }{n},b\in BMO\left( 
%TCIMACRO{\U{211d} }%
%BeginExpansion
\mathbb{R}
%EndExpansion
^{n}\right) $ and $w\in A_{p,q}$. Then the inequality%
\begin{equation}
\left\Vert I_{\alpha ,b}f\right\Vert _{L^{q,w^{q}}\left( B(x,r)\right)
}\lesssim \left\Vert b\right\Vert _{\ast }\left\Vert w\right\Vert
_{L^{q}\left( B(x,r)\right) }\dint\limits_{2r}^{\infty }\ln \left( e+\frac{t%
}{r}\right) \left\Vert f\right\Vert _{L^{p,w^{p}}\left( B\left( x,t\right)
\right) }\left\Vert w\right\Vert _{L^{q}\left( B\left( x,t\right) \right)
}^{-1}\frac{dt}{t}  \tag{$2.5$}
\end{equation}%
holds for any ball $B(x,r)$ and for all $f\in L_{loc}^{p,w^{p}}\left( 
%TCIMACRO{\U{211d} }%
%BeginExpansion
\mathbb{R}
%EndExpansion
^{n}\right) .$
\end{lemma}

\section{Main Results}

\subsection{Boundedness of Fractional Integral operators and their
commutators on vanishing generalized weighted Morrey space}

\begin{theorem}
Let $0<\alpha <n,$ $1<p<\frac{n}{\alpha },$ $\frac{1}{q}=\frac{1}{p}-\frac{%
\alpha }{n},$ $w\in A_{p,q}$, $\varphi \in \mathcal{B}\left( w^{p}\right) $
and $\psi \in $ $\mathcal{B}\left( w^{q}\right) $. Then the operators $%
M_{\alpha }$ and $I_{\alpha }$ are bounded from the vanishing space $VL_{\Pi
}^{p,\varphi }\left( 
%TCIMACRO{\U{211d} }%
%BeginExpansion
\mathbb{R}
%EndExpansion
^{n};w^{p}\right) $ to another vanishing space $VL_{\Pi }^{q,\psi }\left( 
%TCIMACRO{\U{211d} }%
%BeginExpansion
\mathbb{R}
%EndExpansion
^{n};w^{q}\right) ,$ if%
\begin{equation}
c_{\delta }:=\dint\limits_{\delta }^{\infty }\sup_{x\in \Pi }\frac{\varphi ^{%
\frac{1}{p}}\left( x,t\right) }{\left\Vert w\right\Vert _{L^{q}\left(
B\left( x,t\right) \right) }}\frac{dt}{t}<\infty  \tag{$3.1$}
\end{equation}%
for every $\delta >0,$ and%
\begin{equation}
\dint\limits_{r}^{\infty }\frac{\varphi ^{\frac{1}{p}}\left( x,t\right) }{%
\left\Vert w\right\Vert _{L^{q}\left( B\left( x,t\right) \right) }}\frac{dt}{%
t}\leq c_{0}\frac{\psi ^{\frac{1}{q}}\left( x,r\right) }{\left\Vert
w\right\Vert _{L^{q}\left( B\left( x,r\right) \right) }}  \tag{$3.2$}
\end{equation}%
where $c_{0}$ does not depend on $x\in \Pi $ and $r>0.$
\end{theorem}

\begin{proof}
By the pointwise inequality\textbf{\ }$M_{\alpha }f(x)\leq I_{\alpha }\left(
\left\vert f\right\vert \right) (x)$ for $0<\alpha <n$ it is sufficient to
prove the theorem only for fractional integral operator. For every $f\in
VL_{\Pi }^{p,\varphi }\left( 
%TCIMACRO{\U{211d} }%
%BeginExpansion
\mathbb{R}
%EndExpansion
^{n};w^{p}\right) $ we have to show that%
\begin{equation*}
\left\Vert I_{\alpha }f\right\Vert _{VL_{\Pi }^{q,\psi }\left( w^{q}\right)
}\leq c\left\Vert f\right\Vert _{VL_{\Pi }^{p,\varphi }\left( w^{p}\right) }
\end{equation*}%
and%
\begin{equation*}
\underset{r\rightarrow 0}{\lim }\sup\limits_{x\in \Pi }\frac{1}{\psi ^{\frac{%
1}{q}}\left( x,r\right) }\left\Vert I_{\alpha }f\right\Vert
_{L^{q,w^{q}}\left( B\left( x,r\right) \right) }=0.
\end{equation*}%
Firstly, we prove that $I_{\alpha }f\in VL_{\Pi }^{q,\psi }\left(
w^{q}\right) $ when $f\in VL_{\Pi }^{p,\varphi }\left( w^{p}\right) $.
Namely,%
\begin{equation*}
\underset{r\rightarrow 0}{\lim }\sup\limits_{x\in \Pi }\frac{1}{\varphi ^{%
\frac{1}{p}}\left( x,r\right) }\left\Vert f\right\Vert _{L^{p,w^{p}}\left(
B\left( x,r\right) \right) }=0\Rightarrow \underset{r\rightarrow 0}{\lim }%
\sup\limits_{x\in \Pi }\frac{1}{\psi ^{\frac{1}{q}}\left( x,r\right) }%
\left\Vert I_{\alpha }f\right\Vert _{L^{q,w^{q}}\left( B\left( x,r\right)
\right) }=0.
\end{equation*}%
Let $0<r<\delta _{0}$. To show that $\sup\limits_{x\in \Pi }\frac{1}{\psi ^{%
\frac{1}{q}}\left( x,r\right) }\left\Vert I_{\alpha }f\right\Vert
_{L^{q,w^{q}}\left( B\left( x,r\right) \right) }<\varepsilon $ for $%
0<r<\delta _{0}$, by using the inequality $\left( 2.1\right) $ we can write%
\begin{equation}
\frac{1}{\psi ^{\frac{1}{q}}\left( x,r\right) }\left\Vert I_{\alpha
}f\right\Vert _{L^{q,w^{q}}\left( B\left( x,r\right) \right) }\leq C\left[
A_{\delta _{0}}\left( x,r\right) +B_{\delta _{0}}\left( x,r\right) \right] 
\tag{$3.3$}
\end{equation}%
where%
\begin{equation*}
A_{\delta _{0}}\left( x,r\right) :=\frac{\left\Vert w\right\Vert
_{L^{q}\left( B(x,r)\right) }}{\psi ^{\frac{1}{q}}\left( x,r\right) }\left(
\dint\limits_{r}^{\delta _{0}}\frac{\varphi ^{\frac{1}{p}}\left( x,t\right) 
}{t\text{ }\left\Vert w\right\Vert _{L^{q}(B(x,t))}}\underset{0<r<t}{\sup }%
\frac{1}{\varphi ^{\frac{1}{p}}\left( x,r\right) }\left\Vert f\right\Vert
_{L^{p,w^{p}}\left( B\left( x,r\right) \right) }dt\right)
\end{equation*}%
and%
\begin{equation*}
B_{\delta _{0}}\left( x,r\right) :=\frac{\left\Vert w\right\Vert
_{L^{q}\left( B(x,r)\right) }}{\psi ^{\frac{1}{q}}\left( x,r\right) }\left(
\dint\limits_{\delta _{0}}^{\infty }\frac{\varphi ^{\frac{1}{p}}\left(
x,t\right) }{t\text{ }\left\Vert w\right\Vert _{L^{q}(B(x,t))}}\underset{%
0<r<t}{\sup }\frac{1}{\varphi ^{\frac{1}{p}}\left( x,r\right) }\left\Vert
f\right\Vert _{L^{p,w^{p}}\left( B\left( x,r\right) \right) }dt\right) .
\end{equation*}%
Now we fix $\delta _{0}>0$ such that $\sup\limits_{x\in \Pi }\underset{0<r<t}%
{\sup }\frac{1}{\varphi ^{\frac{1}{p}}\left( x,r\right) }\left\Vert
f\right\Vert _{L^{p,w^{p}}\left( B\left( x,r\right) \right) }<\frac{%
\varepsilon }{2Cc_{0}}$, where $c_{0}$ and $C$ are constants from $\left(
3.2\right) $ and $\left( 3.3\right) ,$ respectively. This allows to estimate
uniformly for $0<r<\delta _{0}$%
\begin{equation*}
\sup\limits_{x\in \Pi }CA_{\delta _{0}}\left( x,r\right) <\frac{\varepsilon 
}{2}.
\end{equation*}%
The estimation of the second term, in view the condition (3.1), we get%
\begin{equation*}
B_{\delta _{0}}\left( x,r\right) \leq c_{\delta _{0}}\frac{\left\Vert
w\right\Vert _{L^{q}\left( B(x,r)\right) }}{\psi ^{\frac{1}{q}}\left(
x,r\right) }\left\Vert f\right\Vert _{VL_{\Pi }^{p,\varphi }\left(
w^{p}\right) }
\end{equation*}%
where $c_{\delta _{0}}$ is the constant from $\left( 3.1\right) .$ Since $%
\psi \in \mathcal{B}\left( w^{q}\right) $ it suffices to choose $r$ small
enough such that%
\begin{equation*}
\sup\limits_{x\in \Pi }\frac{\left\Vert w^{q}\right\Vert _{L^{1}\left(
B(x,r)\right) }}{\psi \left( x,r\right) }<\left( \frac{\varepsilon }{%
2Cc_{\delta _{0}}\left\Vert f\right\Vert _{VL_{\Pi }^{p,\varphi }\left(
w^{p}\right) }}\right) ^{q}
\end{equation*}%
which gives required result.

By the definition of the norm and Lemma 2.2 we have%
\begin{eqnarray*}
\left\Vert I_{\alpha }f\right\Vert _{VL_{\Pi }^{q,\psi }\left( w^{q}\right)
} &=&\underset{x\in \Pi ,r>0}{\sup }\frac{1}{\psi ^{\frac{1}{q}}\left(
x,r\right) }\left\Vert I_{\alpha }f\right\Vert _{L^{q,w^{q}}\left( B\left(
x,r\right) \right) } \\
&\leq &c\underset{x\in \Pi ,r>0}{\sup }\frac{1}{\psi ^{\frac{1}{q}}\left(
x,r\right) }\left\Vert w\right\Vert _{L^{q}\left( B(x,r)\right)
}\dint\limits_{r}^{\infty }\left\Vert f\right\Vert _{_{L^{p,w^{p}}\left(
B\left( x,t\right) \right) }}\left\Vert w\right\Vert _{L^{q}\left( B\left(
x,t\right) \right) }^{-1}\frac{dt}{t}.
\end{eqnarray*}%
Thus by the condition (3.2) we get%
\begin{equation*}
\left\Vert I_{\alpha }f\right\Vert _{VL_{\Pi }^{q,\psi }\left( w^{q}\right)
}\leq c\left\Vert f\right\Vert _{VL_{\Pi }^{p,\varphi }\left( w^{p}\right) }
\end{equation*}%
which completes the proof of Theorem 3.2.
\end{proof}

\begin{remark}
Note that in the case $w\equiv 1$ Theorem 3.1 was proved in \cite{Samko}. 
\end{remark}

\begin{theorem}
Let $0<\alpha <n,$ $1<p<\frac{n}{\alpha },$ $\frac{1}{q}=\frac{1}{p}-\frac{%
\alpha }{n},b\in BMO\left( 
%TCIMACRO{\U{211d} }%
%BeginExpansion
\mathbb{R}
%EndExpansion
^{n}\right) $, $w\in A_{p,q}$ and $\varphi \in B(w^{p})$ and $\psi \in
B(w^{q})$. If%
\begin{equation}
c_{\delta }:=\underset{0<r<\delta }{\sup }\text{ }\dint\limits_{\delta
}^{\infty }\left( e+\ln \frac{t}{r}\right) \sup_{x\in \Pi }\frac{\varphi ^{%
\frac{1}{p}}\left( x,t\right) }{\left\Vert w\right\Vert _{L^{q}\left(
B(x,t)\right) }}\frac{dt}{t}<\infty  \tag{$3.4$}
\end{equation}%
for every $\delta >0,$ and%
\begin{equation}
\dint\limits_{r}^{\infty }\ln \left( e+\frac{t}{r}\right) \frac{\varphi ^{%
\frac{1}{p}}\left( x,t\right) }{\left\Vert w\right\Vert _{L^{q}\left(
B\left( x,t\right) \right) }}\frac{dt}{t}\leq c_{0}\frac{\psi ^{\frac{1}{q}%
}\left( x,r\right) }{\left\Vert w\right\Vert _{L^{q}\left( B\left(
x,r\right) \right) }}  \tag{$3.5$}
\end{equation}%
where $c_{0}$ does not depend on $x\in \Pi $ and $r>0,$ then $I_{\alpha ,b}$
is bounded from $VL_{\Pi }^{p,\varphi }\left( w^{p}\right) $ to $VL_{\Pi
}^{q,\psi }\left( w^{q}\right) .$
\end{theorem}

\begin{proof}
Proof of Theorem 3.3 follows by Lemma 2.3 and the same procedure argued in
the proof of Theorem 3.1.
\end{proof}

\begin{corollary}
Let $0<\alpha <n,$ $1<p<\frac{n}{\alpha },$ $\frac{1}{q}=\frac{1}{p}-\frac{%
\alpha }{n},b\in BMO\left( 
%TCIMACRO{\U{211d} }%
%BeginExpansion
\mathbb{R}
%EndExpansion
^{n}\right) $, $w\in A_{p,q}$ and $\varphi \in B(w^{p})$ and $\psi \in
B(w^{q})$. If $\varphi $ and $\psi $ satisfy conditions (3.4) and (3.5) then 
$M_{\alpha ,b}$ is bounded from $VL_{\Pi }^{p,\varphi }\left( w^{p}\right) $
to $VL_{\Pi }^{q,\psi }\left( w^{q}\right) .$
\end{corollary}

\begin{example}
If $w$ is a weight and there exists two positive constants $C$ and $D$ such
that $C\leq w(x)\leq D$ for a.e. $x\in 
%TCIMACRO{\U{211d} }%
%BeginExpansion
\mathbb{R}
%EndExpansion
^{n}$, then obviously $w\in A_{p,q}$ for $1<p<q<\infty .$ Let $0<\alpha <n,$ 
$1<p<\frac{n}{\alpha }$ and $0<\lambda <n-\alpha p.$ Moreover, let $\frac{1}{%
q}=\frac{1}{p}-\frac{\alpha }{n},$ $\varphi (x,r)=r^{\lambda }$ and $\psi
(x,r)=r^{\frac{\lambda q}{p}}.$ Then $M_{\alpha }$ and $I_{\alpha }$ are
bounded from the vanishing space $VL_{\Pi }^{p,\varphi }\left( 
%TCIMACRO{\U{211d} }%
%BeginExpansion
\mathbb{R}
%EndExpansion
^{n};w^{p}\right) $ to another vanishing space $VL_{\Pi }^{q,\psi }\left( 
%TCIMACRO{\U{211d} }%
%BeginExpansion
\mathbb{R}
%EndExpansion
^{n};w^{q}\right) .$
\end{example}

\begin{example}
Let $0<\alpha <n,$ $1<p<\frac{n}{\alpha }$ and $0<\lambda <n-\alpha p.$
Moreover, let $\frac{1}{q}=\frac{1}{p}-\frac{\alpha }{n},$ $\varphi
(x,r)=r^{\lambda }$ and $\psi (x,r)=r^{\frac{\lambda q}{p}}.$ Suppose that $%
v(y)=\left\vert y\right\vert ^{\beta }.$ Then $v\in A_{p}$, $1<p<\infty ,$
if $0<\beta <n(p-1).$ From Lemma 2.1 we have $w(y)=\left\vert y\right\vert ^{%
\frac{\beta }{q}}\in A_{p,q}$, $1<p<q<\infty $ if $0<\beta <n\frac{q}{%
p^{\prime }}$. Thus we can obtain the following:

\begin{equation*}
\left\Vert w\right\Vert _{L^{q}(B(x,r))}\approx \left\{ 
\begin{array}{ccc}
r^{\frac{n+\beta }{q}} & \text{when} & \left\vert x\right\vert <3r \\ 
r^{\frac{n}{q}}(\left\vert x\right\vert +r)^{\frac{\beta }{q}} & \text{when}
& \left\vert x\right\vert \geq 3r%
\end{array}%
\right. .
\end{equation*}

If $\left\vert x\right\vert <3r,$ then we get for the left hand side of (3.2)

\begin{equation*}
\dint\limits_{r}^{\infty }\frac{\varphi ^{\frac{1}{p}}\left( x,t\right) }{%
\left\Vert w\right\Vert _{L^{q}\left( B\left( x,t\right) \right) }}\frac{dt}{%
t}\lesssim r^{\frac{\lambda }{p}-\frac{n+\beta }{q}}
\end{equation*}

since $\frac{\lambda }{p}-\frac{n+\beta }{q}<0.$ On the other hand, for the
right hand side of the inequality (3.2), we have

\begin{equation*}
\frac{\psi ^{\frac{1}{q}}\left( x,r\right) }{\left\Vert w\right\Vert
_{L^{q}\left( B\left( x,r\right) \right) }}\approx r^{\frac{\lambda }{p}-%
\frac{n+\beta }{q}}.
\end{equation*}

If $\left\vert x\right\vert \geq 3r,$ then

\begin{equation*}
\dint\limits_{r}^{\infty }\frac{\varphi ^{\frac{1}{p}}\left( x,t\right) }{%
\left\Vert w\right\Vert _{L^{q}\left( B\left( x,t\right) \right) }}\frac{dt}{%
t}\lesssim \frac{r^{\frac{\lambda }{p}-\frac{n}{q}}+\left\vert x\right\vert
^{\frac{\lambda }{p}-\frac{n}{q}}}{\left\vert x\right\vert ^{\frac{\beta }{q}%
}}
\end{equation*}

and

\begin{equation*}
\frac{\psi ^{\frac{1}{q}}\left( x,r\right) }{\left\Vert w\right\Vert
_{L^{q}\left( B\left( x,r\right) \right) }}\approx \frac{r^{\frac{\lambda }{p%
}-\frac{n}{q}}}{(\left\vert x\right\vert +r)^{\frac{\beta }{q}}}.
\end{equation*}

Obviously

\begin{equation*}
\frac{r^{\frac{\lambda }{p}-\frac{n}{q}}+\left\vert x\right\vert ^{\frac{%
\lambda }{p}-\frac{n}{q}}}{\left\vert x\right\vert ^{\frac{\beta }{q}}}%
\lesssim \frac{r^{\frac{\lambda }{p}-\frac{n}{q}}}{(\left\vert x\right\vert
+r)^{\frac{\beta }{q}}},\qquad \text{when }\left\vert x\right\vert \geq 3r.
\end{equation*}

Furthermore we have%
\begin{equation*}
c_{\delta }\approx \delta ^{\frac{\lambda }{p}-\frac{n+\beta }{q}}.
\end{equation*}

Therefore $M_{\alpha }$ and $I_{\alpha }$ are bounded from the vanishing
space $VL_{\Pi }^{p,\varphi }\left( 
%TCIMACRO{\U{211d} }%
%BeginExpansion
\mathbb{R}
%EndExpansion
^{n};w^{p}\right) $ to another vanishing space $VL_{\Pi }^{q,\psi }\left( 
%TCIMACRO{\U{211d} }%
%BeginExpansion
\mathbb{R}
%EndExpansion
^{n};w^{q}\right) .$
\end{example}

$\hspace{-0.1cm}^{a}$Department of Mathematics,

Faculty of Science, Dicle University,

21280 Diyarbakir, Turkey

\smallskip

$\hspace{-0.1cm}^{b}$Department of Mathematics,

Institute of Natural and Applied Sciences,

Dicle University,

21280 Diyarbakir, Turkey

\end{document}